\documentclass[12pt]{amsart}
\usepackage{amssymb,verbatim,enumerate}
\usepackage{amsmath}
\usepackage{algpseudocode}

\usepackage[mathscr]{eucal}
\usepackage[utf8]{inputenc}
\usepackage[T1]{fontenc}
\usepackage{cite}

\makeatletter
\@namedef{subjclassname@2010}{%
  \textup{2020} Mathematics Subject Classification}
\makeatother

\newcommand{\eN}{{\bar \N}}
\newcommand{\eomega}{{\bar \omega}}

\newtheorem{thm}{Theorem}
\newtheorem{cor}{Corollary}
\newtheorem{prop}{Proposition}
\newtheorem{lem}{Lemma}

\theoremstyle{definition}
\newtheorem{defin}[thm]{Definition}
\newtheorem{rem}{Remark}

\numberwithin{equation}{section}

\def\eq#1{{\rm(\ref{#1})}}
\def\Eq#1#2{\ifthenelse{\equal{#1}{*}}
  {\begin{equation*}\begin{aligned}[]#2\end{aligned}\end{equation*}}
  {\begin{equation}\begin{aligned}[]\label{#1}#2\end{aligned}\end{equation}}}

\newcommand\R{\mathbb{R}}
\newcommand\N{\mathbb{N}}

\newcommand\Z{\mathbb{Z}}
\newcommand\calA{\mathcal{A}}
\newcommand\calE{\mathcal{E}}

\newcommand\HD{d}

 \newcommand{\floor}[1]{\lfloor #1 \rfloor}
\newcommand\XS{\mathcal{X}}

\DeclareMathOperator{\hd}{hd}

\DeclareMathOperator{\Start}{\textsc{Start}}
\DeclareMathOperator{\Move}{\textsc{Move}}

\DeclareMathOperator{\Play}{Play}

\DeclareMathOperator{\pr}{pr}
\DeclareMathOperator{\bit}{bit}
\title{Error recognition in the Cantor cube}
\author{Pawe\l{} Pasteczka}
\thanks{Author is greatful to M.~Boja\'nczyk, E.~Jab{\l}o\'nska and D.~Niwi\'nski for their valuable remarks}
\address{Institute of Mathematics, Pedagogical University of Krak\'ow, Podchor\k{a}\.{z}ych str 2, 30-084 Krak\'ow, Poland}
\email{pawel.pasteczka@up.krakow.pl}
\begin{document}
\begin{abstract}
Based on the notion of thin sets introduced recently by T.~Banakh, Sz.~G\l{}\k{a}b, E.~Jab\l{}o\'nska and J.~Swaczyna we deliver a study of the infinite single-message transmission protocols. Such protocols are associated with a set of admissible messages (i.e. subsets of the Cantor cube $\Z_2^\omega$).  

Using Banach-Mazur games we prove that all 
protocols detecting errors are Baire spaces and generic (in particular maximal) ones are not neither Borel nor meager.

We also show that the Cantor cube can be decomposed to two thin sets which can be considered as the infinite counterpart of the parity bit. This result is related to so-called xor-sets defined by D.~Niwi\'nski and E.~Kopczy\'nski in 2014.
\end{abstract}
\maketitle

\section{Introduction}

We deliver the error-recognition and error-correction approach to single-message transmission protocols which allows to send infinite messages only (i.e. elements of the Cantor cube $\Z_2^\omega$). 
The background of this note are thin sets introduced recently in Banakh--G\l{}\k{a}b--Jab\l{}onska--Swaczyna \cite{BanGlaJabSwa1803}.
%

During our consideration we use all four types of naural numbers (including and excluding both zero and infinity). Thus, in order to avoid misunderstandings, let us define formally
\Eq{*}{
\begin{aligned}[c]
\omega&:=\{0,1,\dots,\}, \\
\N&:=\{1,2,\dots\},
\end{aligned}
\qquad\qquad
\begin{aligned}[c]
\eomega:&=\omega\cup\{\omega\}, \\
\eN:&=\N\cup\{\omega\}.
\end{aligned}
}

Now we recall that the Hamming distance \cite{Ham50} is the function which measures the number of bits where two vectors are different. More precisely we define $\hd\colon \bigcup_{n \in\eN}\Z_2^n \times \Z_2^n \to \eomega$ by
\Eq{*}{
\hd(x, y) := \big|\{k \colon x(k) \ne y(k)\}\big|\qquad x,y \in \Z_2^n,\qquad n \in \eN.
}

One can easy show that $\hd$ is a metric on each $\Z_2^n$ ($n \in \N$) and an extended metric on $\Z_2^\omega$. Therefore we can define the equivalence relation $\sim$ on $\Z_2^\omega$ by 
\Eq{*}{
x \sim y :\iff \hd(x,y)<+\infty.
}
Obviously $\hd$ is a metric on every element of $\Z_2^\omega/_\sim$. 

This notion is deeply connected with the errors recognition, correction and checksums. 
Indeed, once we create a transmission protocol we shall to allow some errors during transmission.

\subsection*{Protocol descripton}
This part is based on \cite[section~4]{Ada91b} (with a natural extension to the case $n=\omega$).
In the simplest model there are two nodes and a single one-sided signal transmission, i.e. Alice send a single message to Bob. Furthermore we assume that number of bits of all admissible messages are the same (say $n \in \eN$). In such a trivial setup whole protocol can be described by a set of all admissible messages which Alice may send to Bob, that is the set $T \subseteq \Z_2^n$.

For the transmission Alice send a message $x \in T$ and Bob receive some element $\bar x \in \Z_2^n$, but due to some distortions in the transmission we do not claim $x=\bar x $.
Obviously if $\bar x  \notin T$ then Bob (or code) \emph{detects} an error. In such a setting he tries to recover a message, i.e. assume that the Alice submit the elements which is closest to $\bar x$ (in the Hamming distance; so-called likehood decoding) -- say $y$ (if the are more than one such element Bob fails to recover). In general we say that Bob (or code) \emph{corrects} the original signal if $y=x$.

As in most of protocol we assume that $\bar x$ is close to $x$, it is reasonable to define an $d$-neighbourhood of $x$, that is the set 
\Eq{*}{
B_d(x):=\{y \in T \colon \hd(y,x)\le d\}.
}
In such a setup Bob can detect up to $d\in \N$ errors  in a code word if and only if $B_d(x) \cap T =\{x\}$ for all $x \in T$. Moreover he can correct up to $d \in \N$ errors (i.e. recovered all messages with at most $d$ errors) if and only if 
$B_d(x)\cap B_d(y)=\emptyset$ for all $x,y \in T$ with $x \ne y$.

Motivated by this fact, for a subset $T \subset \Z_2^n$ ($n \in \eN$) having at least two elements, following  
\cite[section~4.5]{Ada91b} we define 
the \emph{minimum distance} (of $T$) by
\Eq{*}{
\HD(T):=\inf\big\{\hd(x,y) \colon x,y \in T,\, x\ne y\big\}.
}

Then, using some elementary geometrical
argumentation containing in \cite{Ada91b} we can establish the following folk result. 
\begin{lem}\label{lem:detcor}
 Let $T \subset \Z_2^n$ ($n \in\eN$). 
Then, for all $k \in \N$,
\begin{enumerate}[a)]
 \item a code $T$ detects $k$ errors if and only if $\HD(T)\ge k+1$; 
 \item a code $T$ corrects $k$ errors if and only if $\HD(T)\ge 2k+1$.
\end{enumerate}
 \end{lem}
 
Let us emphasize that in the case $n=\omega$ it could happen that $\HD(T)=+\infty$. In this setting the latter lemma states that code detect and corrects any (finite) numbers of errors. Furthermore it shows that the problem of detection, correction, and the minimum distance are in some sense equivalent to each other.

\subsection*{Thin sets} 
A subset $T$ of the Cantor cube $\Z_2^\omega$ is called \emph{thin} if for every number $n \in \omega$ the restriction $\pr_n|T$ of the projection $\pr_n \colon \Z_2^\omega \to 2^{\omega \setminus \{n\}}$ given by $\pr_n \colon  x \mapsto x\vert_{ \omega \setminus \{n\}}$ is injective. Equivalently the minimal distance of $T$ equals at least two -- see Lemma~\ref{lem:ThinEq} for the precise wording of this statement. In view of Lemma~\ref{lem:detcor}, this is a necessary and sufficient condition for a family of infinite streams of bits which allows to detect a single error. In this sense thin sets are the infinite counterpart of a parity bit.

Some properties of this family has been already given in \cite{BanGlaJabSwa1803}. In particular it is known \cite[Proposition~9.3]{BanGlaJabSwa1803} that each Borel thin subset of the Cantor cube is meager and has Haar measure zero. We deliver some further properties of thin sets. In turns out that they are deeply connected with Banach-Mazur games. The key results are obtained using the folk ``capture-the-strategy'' idea.

We also study some special subtype of this family, so-called xor-sets introduced by Niwi\'nski--Kopczy\'nski \cite{NiwKop14}. This allows us to prove that Cantor cube can be partitioned into two thin sets (see section~\ref{sec:xor} for details).

\section{Auxiliary results}

\subsection{Few properties of thin sets} 
First, as a straightforward implication of the definition, we can prove that a subset of a thin set is also thin. Furthermore this family is closed under union of chains. These properties follows from analogous asserts of injective mappings (understood as sets of pairs with suitable assumptions).


\begin{lem}\label{lem:ThinEq}
Let $T \subset \Z_2^\omega$. The following statements are equivalent:
\begin{enumerate}[(i)]
 \item $T$ is thin;
 \item every class of $T /_\sim$ is thin;
 \item $\HD(T) \ge 2$.
 \end{enumerate}
 \end{lem}

 \begin{proof}
Implication $(i) \Rightarrow (ii)$ is obvious as every subset of a thin set is thin. 

To prove $(ii) \Rightarrow (iii)$ assume to the contrary that $\hd(x,y)=1$ for some $x,y \in T$. Then $x\sim y$ and $\{n\in \omega \colon x(n)\ne y(n)\}=\{n_0\}$ for some $n_0 \in \omega$. Thus we have $\pr_{n_0}(x)=\pr_{n_0}(y)$. By the definition of thin set this implies $x=y$, and consequently $\hd(x,y)=0$ contradicting the assumption.

  To show $(iii) \Rightarrow (i)$ assume that $T$ is not a thin set. Then there exists $n_0 \in \omega$ and two distinct elements $x,y \in T$ such that $\pr_{n_0}(x)=\pr_{n_0}(y)$.  Then $\{n \in \omega \colon x(n)\ne y(n)\}=\{n_0\}$, i.e. $\hd(x,y)=1$.
  \end{proof}


By the above results and Zorn Lemma we obtain next proposition.
\begin{prop}\label{prop:3}
 Let $T \subset \Z_2^\omega$ be a thin set. Then there exists a maximal thin set $T_0 \subset \Z_2^\omega$ such that $T \subseteq T_0$. Moreover for all $Q \in \Z_2^\omega/_\sim$ we have that $T_0 \cap Q$ is a maximal thin subset of $Q$.
\end{prop}
\begin{proof}
As thin sets are closed under union of chains, the first part is an immediate implication of Zorn lemma. To show the moreover part assume that there exists $Q \in \Z_2^\omega/_\sim$ such that $T_0\cap Q$ is not a maximal thin subset of $Q$. Then there exists $q \in Q\setminus T_0$ such that
$(T_0\cap Q) \cup \{q\}$ is thin. Therefore applying the implication $(ii)\Rightarrow(i)$ in Lemma~\ref{lem:ThinEq} we obtain that $T_0\cup\{q\}$ is also thin contradicting the maximality.
\end{proof}

\subsection{Banach-Mazur game}
Following Berwanger-Gr\"adel-Kreutzer \cite{BerGraKre03} 
consider a special type of Banach-Mazur game parameterized by a set $F \subset \Z_2^\omega$ (with the product, i.e. Tychonoff topology). Let $\mathcal{G}(F)$ be an infinite two-player game with a complete information, where moves of players consist of selecting and extending finite path through a complete binary tree $\Z_2^\omega$ by an element in $\Z_2^+:=\bigcup_{n=1}^\infty \Z_2^n$. The players will be called Ego and Alter. The two players alternate turns, and each player is aware of all moves before making the next one; Ego begins. All plays are infinite and the result outcome of each play is an element of $x \in \Z_2^\omega$. Ego wins if $x \in F$, otherwise Alter wins. For detailed history of this games we refer the reader to Oxtoby \cite{Oxt57} and Telg\'{a}rsky \cite{Tel87}.

Using some unraveling techniques it is possible to embed $\mathcal{G}(F)$ to the classical Banach-Mazur game on a tree $\Z_2^\omega$ (see \cite{BerGraKre03} for details).
Thus we can reformulate the original Banach-Mazur theorem \cite{Ban32} in the flavour of Berwanger-Gr\"adel-Kreutzer. Prior to this we need to recall the notion of strategy.

Ego's and Alter's strategy are the functions 
\Eq{*}{
e \colon \bigcup_{n=0}^\infty (\Z_2^+)^n\to \Z_2^+\qquad\text{ and  }\qquad a \colon \bigcup_{n=1}^\infty (\Z_2^+)^n\to \Z_2^+,
}
respectively. Denote sets of all Alter's and Ego's strategies by $\calA$ and $\calE$. For $a \in \calA$ and $e \in \calE$ one can consider a sequence of moves 
\begin{align*}
 \epsilon_0&:=e(\emptyset), \\
 \alpha_i&:=a(\epsilon_0,\dots,\epsilon_{i-1}), &\qquad i \in \N_+ \\
 \epsilon_i&:=e(\alpha_1,\dots,\alpha_i), &\qquad i \in \N_+ 
\end{align*}
and define a \emph{play} (hereafter we use the classical abbreviation to concatenation of sequences)
\Eq{*}{
\Play\colon \calE\times \calA \ni (e,a) \mapsto (\epsilon_0\alpha_1\epsilon_1\alpha_2\epsilon_2\alpha_3\dots)\in \Z_2^\omega.
}
This is very usual notion in game theory -- instead of sequence of moves players show whole strategy at the beginning. We also use the time-lapse approach to a strategy. We treat it as a sequence of replies for the opponent's moves and write it in terms of pseudocode -- it is a classical approach in game theory which is equivalent to the one above.

We say that $a_0 \in \calA$ is \emph{an Alter's winning strategy} (in $\mathcal{G}(F)$) if $\Play(e,a_0) \notin F$ for all $e \in \calE$. Analogously $e_0 \in \calE$ is \emph{an Ego's winning strategy} (in $\mathcal{G}(F)$) if $\Play(a,e_0) \in F$ for all $a \in \calA$. If one of players has a winning strategy then the game $\mathcal{G}(F)$ is \emph{determined}. Now we can recall celebrated Banach-Mazur theorem.

\begin{thm}[Banach-Mazur]Let $F \subset \Z_2^\omega$.
\begin{enumerate}
 \item  Alter has a winning strategy for the game $\mathcal{G}(F)$ if and only if $F$ is meager.
\item  Ego has a winning strategy for the game $\mathcal{G}(F)$ if and only if there exists finite word $x \in \Z_2^+$ such that $(x\cdot \Z_2^\omega)\setminus F$ is meager.
 \end{enumerate}
\end{thm}

As a result we have the following corollary. 
\begin{cor}\label{cor:detBor}
Games $\mathcal{G}(F)$ are determined for all Borel sets $F$. 
\end{cor}

In the following two propositions we present necessary conditions for Ego and Alter to have a winning strategy. The first result essentially follows the idea of Niwi\'nski and Kopczy\'nski from \cite{NiwKop14}.

\begin{prop}\label{prop:NessPro}
Let $F \subseteq \Z_2^\omega$ be a thin set. Then Ego has no winning strategy in a game $\mathcal{G}(F)$.
\end{prop}
\begin{proof}
Assume to the contrary that Ego has a winning strategy in $\mathcal{G}(F)$. We play this game two times simultaneously -- we call them ``initial'' and ``mirror'' play.
Denote the Ego's moves in the initial and mirror plays as $(\alpha_i)_{i=0}^\infty$ and $(\beta_i)_{i=0}^\infty$, respectively. 
Obviously $\alpha_0=\beta_0$ as the first Ego's move is fixed.

First Alter's reply in the initial play is $0$. In the mirror play it is $(1\alpha_1)$. From now on Alter capture the Ego's strategy in the following way:
\begin{itemize}
 \item each time Ego moves $\alpha_k$ ($k \ge 2$) in the initial play, Alter copy this move to the mirror play;
 \item each time Ego moves $\beta_k$ ($k \ge 1$) is the mirror play, Alter copy this move to the initial play.
\end{itemize}

These two plays can be illustrated in the table-like form (Simulation~\ref{sim:xor})

\begin{center}
\begin{figure}[!h]
 \begin{tabular}{l|cccccccccccc} 
\hline
    \multicolumn{13}{|c|}{Initial play} \\ 
\hline
      Ego &  $\alpha_0$ & & $\alpha_1$ & & $\alpha_2$ &  &$\dots$ & $\alpha_k$ & &$\dots$&\\
      Alter & & $0$ & & $\beta_1$ & & $\beta_2$& $\dots$ & &  $\beta_k$ & $\dots$ & &\\
\hline
\hline
    \multicolumn{13}{|c|}{Mirror play} \\ 
      \hline
      Ego &$\alpha_0$  &\multicolumn{2}{c}{} & $\beta_1$ & & $\beta_2$ & $\dots$ & &$\beta_k$ &$\dots$\\
      Alter & & \multicolumn{2}{c}{$1\alpha_1$}& & $\alpha_2$ & & $\dots$ &  $\alpha_k$ & & $\dots$ \\
    \end{tabular} 
\caption{\label{sim:xor}Capturing Ego's strategy}
\end{figure}
\end{center}

The final outcome of the initial and mirror plays are
$$a:=(\alpha_00\alpha_1\beta_1\alpha_2\beta_2\cdots) \text{ and } b:=(\alpha_01\alpha_1\beta_1\alpha_2\beta_2\cdots),$$
respectively. As Ego has a winning strategy in $\mathcal{G}(F)$ then we obtain $a,b \in F$. However in this case we have
$\hd(a,b)=1$. This implies that $F$ is not a thin set, contradicting the assumptions.
 \end{proof}

 Krom \cite{Kro74}  proved that Ego has no winning strategy in $\mathcal{G}(F)$ if and only if $F \subset \Z_2^\omega$ is the Baire space. Thus applying Proposition~\ref{prop:NessPro} we immediately obtain that thin sets are Baire spaces.

Now we are heading toward the necessary condition to Alter's winning strategy, however we need to introduce few notions first. For $k \in \omega$ define the function $\bit_k \colon \omega \to \{0,1\}$ such that $\bit_k(x)$ is the $k$-th bit from the right in the binary notation of $x$ (counting from zero). More precisely, for $k, n \in \omega$ we have
$$
\bit_k(n) = \begin{cases}
  0 & \text{ if } (n \bmod 2^{k+1}) \in \{0,\dots,2^k-1 \}, \\
  1 & \text{ if } (n \bmod 2^{k+1}) \in \{2^{k},\dots,2^{k+1}-1\}.           
            \end{cases}
$$

For $n \in \N$, $x\in \Z_2^n$ and $m \in \{0,1,\dots,2^n-1\}$ define $\Theta(x,m)\in \Z_2^n$ as follows ($\oplus$ stands for a binary xor)
$$\Theta(x,m):=(x_k \oplus \bit_k(m))_{k\in\{0,\dots,n-1\}}.$$

For an infinite sequence $x \in \Z_2^\omega$ and $m \in \omega$ define $\Theta(x,m)\in \Z_2^\omega$ by   
$$\Theta(x,m):=(x_k \oplus \bit_k(m))_{k\in\omega}.$$ 

We can now proceed to formulate and proof the most technical proposition of this paper.
 
\begin{prop}\label{prop:NecVer}
Let $F \subseteq \Z_2^\omega$ be a set such that Alter has a winning strategy in a game $\mathcal{G}(F)$. Then there exists an element $X \in \Z_2^\omega /_\sim$ such that $X\cap F=\emptyset$.
\end{prop}

\begin{proof}

Fix an Alter strategy. Now we consider infinitely many plays of $\mathcal{G}(F)$ and show that their output covers whole class of abstraction of $\sim$. 

Let $(v_i)_{n=1}^\infty$ of elements in $\Z_2^+$ be a sequence of Alter replies (in all plays), enumerated by the order of moves. Ego spreads Alters replies among all plays in a way which are described by the algorithm below.
There are two types of Ego's moves: $\Start_i(\alpha)$ and $\Move_i(\alpha)$ for $i \in \omega$ and $\alpha \in \Z_2^+$:
\begin{enumerate}[1.]
 \item $\Start_i(\alpha)$ -- Ego starts Play~$i$ with the initial move $\alpha$;
\item $\Move_i(\alpha)$ -- Ego makes a subsequent move $\alpha$ in Play~$i$.
\end{enumerate}

We now present the algorithmic description of the Ego's strategy (in infinitely many plays). It depends on the Alter's strategy which is emphasized as an argument (this is a sort of an input stream to this procedure).
\begin{algorithmic}[0] 
\Procedure{Capture}{Alter Strategy $\sigma$}
\State $\Start_0(0)$
\hfill \textit{ Alter replies: }$v_1$
\For{$i=1$ \textbf{ to }$+\infty$} 
\State{$\Start_i\big(\Theta(0v_1\dots v_{\frac{i(i+3)}2-1},i)\big)$}
\hfill \textit{ Alter replies: }$v_{\frac{i(i+3)}2}$
\For{$j=0$ \textbf{ to }$i$}
\State{$\Move_j\big(v_{\frac{i(i+1)}2+1+j}\dots v_{\frac{i(i+1)}2+i+j}\big)$}
\textit{ Alter replies: }$v_{\frac{i(i+3)}2+j+1}$
\EndFor
\EndFor
\EndProcedure
\end{algorithmic}
Similarly to the previous proof let us illustrate several first moves in a tabular form (cf. Simulation~\ref{sim:CAS}).
 
 \begin{figure}
\begin{center}
\begin{tabular}{l|ccccccccccc} 
\hline
    \multicolumn{12}{|c|}{Play 0} \\ 
    \hline
      Ego &  $0$ & & $v_2$ & & \multicolumn{2}{c}{$v_4v_5$} &  &\multicolumn{3}{c}{$v_7v_8v_9$}&\\
      Alter & & $v_1$ & & $v_3$ & & & $v_6$ &\multicolumn{3}{c}{} &  $v_{10}$ \\
\hline
\hline
    \multicolumn{12}{|c|}{Play 1} \\ 
      \hline
      Ego & \multicolumn{2}{c}{$1v_1$} & & $v_3$& & \multicolumn{2}{c}{$v_5v_6$} &  &\multicolumn{3}{c}{$v_8v_9v_{10}$} \\
      Alter &  &  & $v_2$ & & $v_4$ & & & $v_7$ &\multicolumn{2}{c}{}&\dots \\
\hline
\hline
    \multicolumn{12}{|c|}{Play 2} \\ 
      \hline
      Ego & \multicolumn{5}{c}{$\Theta(0v_1v_2v_3v_4,2)$} & & \multicolumn{2}{c}{$v_6v_7$} && \multicolumn{2}{c}{\dots}\\
      Alter &  \multicolumn{5}{c}{} & $v_5$ & & & $v_8$ \\
\hline
\hline
    \multicolumn{12}{|c|}{Play 3} \\ 
      \hline
      Ego & \multicolumn{9}{c}{$\Theta(0v_1v_2v_3v_4v_5v_6v_7v_8,3)$} & & \dots \\
      Alter &  \multicolumn{9}{c}{} & $v_9$ \\
    \hline
    \end{tabular} \\ \medskip
    \dots
\end{center}
\caption{\label{sim:CAS}Capturing Alter's strategy}
 \end{figure}

Obviously both players make infinitely many moves in each of plays. 
Furthermore, in order to show that this algorithm is correct, we need to show that Alter's replies are properly enumerated (i.e. id of the element coincide with the replies number). This proof is a straightforward application of the ``loop invariant'' method; for the details we refer the reader to the classical book \cite{CorLeiRivSte09}.

Now let $r_i \in \Z_2^\omega$ be the output of Play~$i$ ($i\in\omega$).
As Ego rewrites all Alter answers except the initial move we obtain 
\Eq{*}{
r_i=\Theta(r_0,i)\quad \text{ for all }i\in \omega.
}

Therefore if Alter has a winning strategy we get $\{r_i\colon i \in \omega\} \cap F =\emptyset$, and thus $[r_0]_\sim \cap F=\emptyset$.
\end{proof}

 \section{Main result}
  
  In this brief section we present three results. First of them is presented in the game setting approach, while second and third one are topological properties of thin sets.
  
 \begin{prop}
  Let $T \subset \Z_2^\omega$ be a thin set such that $T\cap X \ne \emptyset$ for all $X \in \Z_2^\omega/_\sim$. Then the game $\mathcal{G}(T)$ is undetermined.
 \end{prop}
\begin{proof}
Since $T$ is thin, Proposition~\ref{prop:NessPro} implies that Ego has no winning strategy in $\mathcal{G}(T)$.
On the other hand as $T\cap X \ne \emptyset$ for all $X \in \Z_2^\omega/_\sim$ by Proposition~\ref{prop:NecVer} we get that Alter has no winning strategy in $\mathcal{G}(T)$, too. 
\end{proof}

Now, applying Corollary~\ref{cor:detBor}, we can formulate the main result of this paper. 
 \begin{thm}\label{thm:NB}
  If $T \subset \Z_2^\omega$ is a thin set such that $T\cap X \ne \emptyset$ for all $X \in \Z_2^\omega/_\sim$ then $T$ is not Borel.
 \end{thm}

As a singleton is a thin set, by Proposition~\ref{prop:3} we easily obtain
\begin{cor}\label{cor:maxBor}
Maximal thin sets are neither Borel nor meager. 
 \end{cor}

\section{\label{sec:xor} Applications to Xor-sets}
We show that xor-sets introduced by Niwi\'nski--Kopczy\'nski \cite{NiwKop14} are maximal thin sets. 
Before we go into details, let us introduce some sort of conjugency. For $x \in \Z_2^\omega$ and $n \in \omega$ we define $x^{\#n}\in \Z_2^\omega$ by 
$$x^{\#n} (k):=\begin{cases} 
x(k) & k \in \omega \setminus \{n\},\\
1-x(n) &k=n.
\end{cases}$$
For fixed $n \in \omega$ the operator $(\cdot)^{\#n}$ is a symmetry, i.e. $(x^{\#n})^{\#n}=x$ for all $x \in \Z_2^\omega$. We are now in the position to present the main definition of this section. 

\begin{defin}
 A set $\XS \subset \Z_2^\omega$ is called a \emph{xor-set} if for every $n \in \omega$ and $x \in \Z_2^\omega$ we have
 $x \in \XS \iff x^{\#n} \notin \XS$.
 \end{defin}

Since $(\cdot)^{\#n}$ is a symmetry we can easily check that for every xor-set $\XS$ the set $\Z_2^\omega \setminus \XS$ is a xor-set, too. We prove that xor-sets are maximal thin sets. Before we go into details let us introduce few technical notions. 

First, let us define the relation $\approx$ on $\Z_2^\omega$ by 
$$x \approx y :\iff \hd(x,y)\text{ is finite and even.}$$

Observe that $x \approx y$ implies $x \sim y$. Moreover each element of $\Z_2^\omega/_\sim$ split into two elements of $\Z_2^\omega/_\approx$. 
Therefore, as $\Z_2^\omega/_\sim$ has a cardinality continuum, one consider a partition of $\Z_2^\omega$ to a family of disjoint sets $\mathcal{U}:=\{U_{\eta,j}\colon \eta \in \R\text{ and }j \in \{0,1\}\}$ such that $U_{\eta,0},U_{\eta,1} \in \Z_2^\omega/_\approx$ and $U_{\eta,0}\cup U_{\eta,1} \in \Z_2^\omega/_\sim$ ($i \in\R$). These sets play an essential role in the theory of xor-sets. 

\begin{lem}
Set $\XS \subset \Z_2^\omega$ is a xor-set if and only if there exists a selector $\mathcal{S}$ of $\big\{\{U_{\eta,0},U_{\eta,1}\} \colon \eta \in \R\big\}$  such that $\XS =\bigcup \mathcal{S}$.
\end{lem}
\begin{proof}
Let $\eta \in \R$. Observe that for all $x \in U_{\eta,0}\cup U_{\eta,1}$ and $n \in \omega$ we have $x^{\#n}\sim x$ and $x^{\#n}\not\approx x$. Therefore $x \in U_{\eta,0} \iff x^{\#n} \in U_{\eta,1}$. 

Furthermore for all $y \in \Z_2^\omega$ such that $y \approx x$ there exists $k \in \N$ and a set $\{i_1,\dots,i_{2k}\}$ such that
$y=(\cdots((x^{\#i_1})^{\#i_2})\cdots)^{\#i_{2k}}$. Whence $x\in \XS$ implies $[x]_\approx \subset \XS$. 

Furthermore $x \in U_{\eta,0} \iff y\in U_{\eta,0}$ and $x \in U_{\eta,1} \iff y\in U_{\eta,1}$. 
Thus $[x]_\approx \in \{U_{\eta,0},U_{\eta,1}\}$. As $[x]_\sim$ contains two classes of abstraction of $\approx$ we get 
$$\big\{U_{\eta,0},U_{\eta,1}\big\}=\big\{[x]_\approx,([x]_\sim \setminus [x]_\approx)\big\}.$$
By the definition every xor-set contains exactly one element of each pair, which implies our assertion. 
\end{proof}

\begin{prop}
Every xor-set $\XS \subset \Z_2^\omega$ is a maximal thin set. In particular, $\XS$ is not Borel.
\end{prop}
\begin{proof}
As $\XS$ a xor-set we have that for all $x,y \in \XS$ the distance $\hd(x,y)$ is either infinite of even. Then the implication $(iii)\Rightarrow(i)$ in Lemma~\ref{lem:ThinEq} yields that $\XS$ is a thin set.

To show the maximality assume to the contrary that there exist a xor-set $\XS$ and an element $x \in \Z_2^\omega \setminus \XS$ such that $\XS \cup\{x\}$ is thin. 
 By the definition of xor-set we have $x^{\#1} \in \XS$. However $\pi_1(x)=\pi_1(x^{\#1})$ which lead to a contradiction as $\XS \cup\{x\}$ was supposed to be thin. 
 
 The remaining part is a straightforward implication of Corollary~\ref{cor:maxBor}.
\end{proof}

As a complementary of a xor-set is a xor-set we obtain the following interesting property.
\begin{cor}\label{cor:X1}
 There exists two non-Borel, thin and disjoint sets $T_0,\ T_1 \subset \Z_2^\omega$ such that $T_0\cup T_1=\Z_2^\omega$.
\end{cor}

In fact we can also prove a sort of the reverse statement
\begin{prop}\label{prop:X2}
 Let $T_0$, $T_1$ be two thin sets such that $T_0\cup T_1=\Z_2^\omega$. Then $T_0$ and $T_1$ are disjoint xor-sets.
\end{prop}
\begin{proof}
 Indeed, as $T_0$ is thin and $T_0\cup T_1=\Z_2^\omega$ we have
\Eq{EH0}{
x \in T_0 \Rightarrow x^{\#n}\notin T_0 \Rightarrow x^{\#n} \in T_1 \qquad \text{for all }x \in \Z_2^\omega\text{ and }n \in \omega.
}
Similarly $x \in T_1 \Rightarrow x^{\#n} \in T_0$, which yields  $x \in T_0 \iff x^{\#n} \in T_1$.

If there existed $x\in T_0\cap T_1$ then by \eq{EH0} we would obtain $x^{\#n} \in T_1$, contradicting the fact that $T_1$ is a thin set. Therefore $T_0\cap T_1=\emptyset$. Then we have 
\Eq{*}{ 
x \in T_i \iff x^{\#n} \notin T_i\qquad \text{ for all }x \in \Z_2^\omega,\ n \in \omega\text{ and }i \in\{0,1\}
}
which shows that both $T_0$ and $T_1$ are xor-sets.
 \end{proof}

\begin{rem}
Applying above results we can easily show that $\XS \subset \Z_2^\omega$ is a xor-set if and only if both $\XS$ and $\Z_2^\omega \setminus \XS$ are thin. 
\end{rem}

\section{Generalization to $k$-thin sets}
At the very end of this section let us just mention the natural generalization of thin sets to $k$-thin sets. Namely, the set 
$T \subset \Z_2^n$ ($n \in\eN$) is called  \emph{$k$-thin} if its minimum distance equals at least $k$. Then $2$-thin sets are precisely thin sets. 

There appears the natural question, if there exists a partition of the Cantor cube to finitely many $k$-thin sets. For $k=2$ the answer easily follows from Proposition~\ref{prop:X2}. We show that it cannot be generalized to $k$-thin sets.

For $k \in \N$ and $n \in \eN$ with $2 \le k \le n$ we define $Q(n,k)$ as the smallest extended natural number $s$ such that there exists a partition of $\Z_2^n$ into $s$ sets which are $k$-thin.

Then we obviously have $Q(n,2)=2$ for all $n \ge 2$. This case turns out to be very special.
\begin{prop}
Let $k \in \N$, $n \in \eN$ with $3 \le k \le n$. Then
\begin{enumerate}
 \item for finite $n$ we have $Q(n,k)\ge {n \choose {\lfloor \frac{k-1}2 \rfloor}}$
 \item for $n =\omega$ we have $Q(\omega,k)=+\infty$.
\end{enumerate}
\end{prop}

\begin{proof}

Consider an arbitrary partition  $(T_1,\dots,T_q)$ of $\Z_2^n$ to $k$-thin sets, where $q \in \eN$. Furthermore set
\Eq{*}{
S:=\big\{x \in \Z_2^n \colon x_0+\dots+x_{n-1} \le
\floor{\tfrac{k-1}{2}}\big\}.}
Then for all $x,y \in S$ we have 
\Eq{*}{
\hd(x,y)\le \hd(x,0)+\hd(0,y) \le 2\floor{\tfrac{k-1}{2}} \le k-1.
}

In particular each $T_i$ contains at most one element of $S$. Thus in the case where $n$ is finite we obtain
$q \ge |S| \ge {n \choose {\floor{\frac{k-1}2}}}$.
For finite number $n$ we simply get $q \ge |S|=+\infty$.
\end{proof}

Finally, observe that we can reapply Theorem~\ref{thm:NB} and Proposition~\ref{prop:3} to generalize Corollary~\ref{cor:maxBor} as follows
\begin{prop}
Maximal $k$-thin sets are neither Borel nor meager. 
 \end{prop}


\end{document}